\documentclass[reqno]{amsart}
\usepackage{amssymb}
\usepackage{amsmath}
\usepackage{amsthm}
\usepackage{mathtools}
\usepackage{graphicx}
\usepackage{comment}

\usepackage{amssymb, mathrsfs, amsfonts, amsmath}

\usepackage{tikz}
\usetikzlibrary{arrows}

\usepackage[english]{babel}

\usepackage[margin=1.4in, centering]{geometry}

\DeclareSymbolFont{bbold}{U}{bbold}{m}{n}
\DeclareSymbolFontAlphabet{\mathbbold}{bbold}
\newtheorem{thm}{Theorem}
\newtheorem{prop}[thm]{Proposition}
\newtheorem{lem}[thm]{Lemma}
\newtheorem{cor}[thm]{Corollary}

\theoremstyle{definition}

\theoremstyle{remark}
\newtheorem{rem}[thm]{Remark}

\newcommand{\R}{\mathbb{R}}
\newcommand{\C}{\mathbb{C}}

\newcommand{\eps}{\varepsilon}

\def\XXint#1#2#3{{\setbox0=\hbox{$#1{#2#3}{\int}$ }
\vcenter{\hbox{$#2#3$ }}\kern-.6\wd0}}

\title{A singular variant of the Falconer distance problem}
\author{Tainara Borges, Alex Iosevich, and Yumeng Ou}

\address[T. Borges]{Department of Mathematics, Brown University, Providence, RI 02912}
\address[A. Iosevich]{Department of Mathematics, University of Rochester, Rochester, NY 14627}
\address[Y. Ou]{Department of Mathematics, University of Pennsylvania, Philadelphia, PA 19104}

\begin{document}

\begin{abstract} In this paper we study the following variant of the Falconer distance problem. Let $E$ be a compact subset of ${\R}^d$, $d \ge 1$, and define 
$$ \Box(E)=\left\{\sqrt{{|x-y|}^2+{|x-z|}^2}: x,y,z \in E,\, y\neq z \right\}.$$

We shall prove using a variety of methods that if the Hausdorff dimension of $E$ is greater than $\frac{d}{2}+\frac{1}{4}$, then the Lebesgue measure of $\Box(E)$ is positive. This problem can be viewed as a singular variant of the classical Falconer distance problem because considering the diagonal $(x,x)$ in the definition of $\Box(E)$ poses interesting complications stemming from the fact that the set $\{(x,x): x \in E\}\subseteq \R^{2d}$ is much smaller than the sets for which the Falconer type results are typically established.

We also prove a finite field variant of the Euclidean results for $\Box(E)$ and indicate both the similarities and the differences between the two settings. 
\end{abstract}

\maketitle

\tableofcontents
\section{Introduction}

Let $E\subseteq \R^{d}$ be a compact set, where $d\geq 1$. $|\cdot|$ will denote the euclidean distance in $\R^{d}$ or $\R^{2d}$ depending on the context. We will be interested in the following special distance set

\begin{equation}\label{boxEdef}
\begin{split}
     \Box(E)=&\{|(y,z)-(x,x)|\colon x,y,z\in E,\, y\neq z\}  \\
     =&\{\sqrt{|y-x|^2+|z-x|^2}\colon x,y,z\in E,\, y\neq z\}
\end{split}  
\end{equation}
and its pinned version 

\begin{equation}
    \Box_{x}(E)=\{|(y,z)-(x,x)|\colon y,z\in E,\, y\neq z\},\quad x\in E.
\end{equation}

Our main motivating questions are the following Falconer type distance problems:

\begin{enumerate}
\item How large does the Hausdorff dimension of $E$ has to be so that the Lebesgue measure $\mathcal{L}^{1}(\Box(E))>0$?  What if we want $\Box(E)$ to contain an interval?
\item  How large does the Hausdorff dimension of $E$ has to be so that there exists $x\in E$ with $\mathcal{L}^{1}(\Box_{x}(E))>0$? When can we guarantee that $\Box_x(E)$ contains an interval?
\end{enumerate}

If one replaces the set $\Box(E)$ (resp. $\Box_x(E)$) by the classical distance set $\Delta(E):=\{|x-y|:\, x,y\in E\}$ (resp. $\Delta_x(E):=\{|x-y|:\, y\in E\}$), the above questions are usually referred to as the Falconer distance problem. The celebrated Falconer distance conjecture \cite{falconer85} says that if $E\subset \mathbb{R}^d$ has Hausdorff dimension ${\rm dim}(E)>\frac{d}{2}$, then $\mathcal{L}^1(\Delta(E))>0$. The conjecture is still open in all dimensions $d\geq 2$ and has been a major driving force behind many recent breakthroughs in harmonic analysis and geometric measure theory. See for instance \cite{GIOW,DIOWZ,DZ19,DGOWWZ, Orponen, SW} for the currently best known partial results and some recent development concerning sets satisfying special regularity. It is also worth noticing that all the state-of-the-art partial results for the Falconer distance problem actually say that the smaller pinned distance set $\Delta_x(E)$ has positive Lebesgue measure, for some $x\in E$.

In this context, the box set defined above can be viewed as a prototype example of a refined version of the distance set $\Delta(E\times E):=\{|(y,z)-(y',z')|:\, (y,z),\,(y',z')\in E\times E\}$ where we only consider distances to the diagonal. The restriction $y\neq z$ is necessary here to make sure that it captures distances from points away from the diagonal. Indeed, for $y=z$, the corresponding box set becomes precisely a rescaled distance set $\Delta(E)$, to which the Falconer distance results already apply. 

Note that this variant of the distance problem is singular in the sense that the subset we are restricting to (in this example, the diagonal) has lower Hausdorff dimension than the set $E\times E$. This poses many challenges to the study of the problem. For instance, let $\mu$ be the standard Frostman measure supported on $E$. Even if one knows that $\mathcal{L}^1(\Delta_{(y,z)}(E\times E))>0$ for all $(y,z)$ in a subset of $E\times E$ with positive $\mu \times \mu$ measure, it wouldn't be enough to yield that $\mathcal{L}^1(\Box_x(E))>0$ for some $x\in E$, since the diagonal $\{(x,x)\}$ has zero $\mu \times \mu$ measure. 

It is also natural to consider more general versions of such singular variants of distance sets, for example distances to other lower dimensional geometric objects in the set. Indeed, the method we introduce in this article extends well beyond the box set, and we will comment on this with more details in the proofs later.

\subsection*{Main results and methods}

Our main theorems are the following. First,

\begin{thm}\label{thm: interval}
Let $d\geq 2$ and $E\subset \mathbb{R}^d$ be a compact set. If its Hausdorff dimension 
\[
{\rm dim}(E)>\begin{cases} \frac{d}{2}+\frac{1}{4},& \text{$d$ even},\\ \frac{d}{2}+\frac{1}{4}+\frac{1}{8d-4},& \text{$d$ odd},
\end{cases}
\]then there is a point $x\in E$ such that $\Box_x(E)$ contains an interval.
\end{thm}

In the case that one is only interested in the Lebesgue measure of $\Box(E)$ and $\Box_x(E)$, the above result can be further improved to the following, which holds even in dimension one.

\begin{thm}\label{thm: measure}
Let $d\geq 1$ and $E\subset \mathbb{R}^d$ be a compact set. If its Hausdorff dimension ${\rm dim}(E)>\frac{d}{2}+\frac{1}{4}$, then the Lebesgue measure $\mathcal{L}^1(\Box(E))>0$.

Moreover, if $d=1$ and ${\rm dim}(E)>\frac{1}{2}+\frac{1}{4}=\frac{3}{4}$, then there is a point $x\in E$ such that $\mathcal{L}^1(\Box_x(E))>0$.
\end{thm}

Our results are among the first efforts to investigate singular variant of Falconer distance problems. While the draft of this paper was in the final stages of preparation, the paper \cite{GPP2023} appeared, where one of our main problems in the Euclidean setting was studied independently. Their argument is based on a different approach, but our dimensional thresholds are considerably less restrictive. 

In order to obtain the above theorems, we introduce four different approaches, each of which has its own unique advantages. One of the main purposes of this article is to showcase how the plethora of methods can be applied to this type of singular variants of the distance problem. 

Our first approach is based on the observation that the pinned box set can be treated as a sum set of classical pinned distance sets. It not only implies Theorem \ref{thm: interval} but can also be used as a black box to turn future improvement for the classical Falconer distance problem to that for the box problem discussed here (see Theorem \ref{sumsettheorem} for the precise statement). Indeed, the dimensional thresholds stated in Theorem \ref{thm: interval} match precisely the best known thresholds for the Falconer distance problem (obtained in \cite{GIOW,DZ19,DIOWZ,DGOWWZ}). Moreover, this approach is powerful enough to show that the box set has non-empty interior. Note that this approach only works in dimension $d\geq 2$, since the Falconer distance problem becomes trivial in $d=1$ (More precisely, there exists set $E\subset \mathbb{R}$ with arbitrarily large Hausdorff dimension such that its distance set has zero Lebesgue measure).

Our next approach is based on a Sobolev norm bound for the bilinear spherical averaging operator (Theorem \ref{sobolevestimates} and \ref{sobolevestimatesforB}), which is the key ingredient in the proof of the first part of Theorem \ref{thm: measure}. The bilinear spherical averaging operator and its associated maximal functions have attracted a great amount of attention in recent years, see for instance \cite{JL,borgesfoster,BFOPZ,PSsparse, MCZZ,dosidisramos} for its $L^p$ improving estimates, $L^p$ norm bounds, sparse bounds, etc. Our Sobolev norm bounds for this operator is among the first of its kind (see \cite{BFOPZ} for another type of such bound in $d=1$ and \cite{GPP2023} for an $L^2\times L^{2}\rightarrow L^2$ variant) and is expected to be of independent interest.

The second part of Theorem \ref{thm: measure} follows from our third approach, which is based on the connection of the box problem to the distance problem concerning two different sets (for example see the discussion on the sharpness example given below). More precisely, the question here can be reduced to understanding the dimensional threshold needed for the join distance set $\Delta(A, B):=\{|x-y|:\, x\in A,\, y\in B\}$ to have positive measure or non-empty interior. This question is interesting in its own right and we will show that such result in certain ranges of dimensions works very effectively in the study of the box problem. Indeed, this approach produces the pinned version of the equally strong result as the above bilinear approach in dimension $d=1$. However, it does not work as well in higher dimensions. We leave the detailed discussion to Section \ref{sec: product set}.

In addition, we introduce another approach relating the box set problem to the chains generated by a given set. Over the past decade, there has been a great amount of interest in understanding the size of the $k$ chain set, similarly as the distance set (which corresponds to the case $k=1$), see for instance \cite{BIT, OT2020} and the references therein. In Section \ref{sec: projection}, we will use a projection method to show that some of the box set results are implied by the corresponding results for the chains. This idea is quite robust and works for many other related problems beyond the box set. We will mention a few such applications in Section \ref{sec: projection}.

It is reasonable to conjecture that the sharp exponent for the $\Box(E)$ problem is $\frac{d}{2}$. To see this let $E_q$ denote the $q^{-\frac{d}{s}}$-neighborhood (where $s>0$ is to be determined later) of 
$$ \frac{1}{q} \left\{ {\mathbb{Z}}^d \cap {[0,q]}^d \right\}.$$

It is not difficult to see (see e.g. \cite{falconer85}) that if $q_i$ is a sequence of positive integers growing super exponentially, then the set $\cap_i E_{q_i}$ has Hausdorff dimension $s$. Let $A=E \times E$ and $B=\{(x,x):x \in E\}$. Let $A_q=E_q \times E_q$ and $B_q=\{(x,x): x \in E_q\}$. We have $\Box(E)=\Delta(A,B)=\{|u-v|: u \in A, v \in B\}$. Observe that 
$$ |\Delta(A_q,B_q)| \lesssim q^{-\frac{d}{s}} \cdot q^2.$$ This expression goes to $0$ if $s<\frac{d}{2}$, establishing $\frac{d}{2}$ as a reasonable conjectured exponent. 

With all three approaches described above, in order to guarantee the non-degeneracy of the box sets, we will need to do an initial reduction. We first record the following simple observation.

\begin{lem}\label{lem: subsets}
Let $d\geq 1$ and $E\subset \mathbb{R}^d$ be a compact set. Suppose $\mu$ is a Frostman measure supported on $E$ satisfying
\begin{equation}\label{eqn: ball}
\mu(B(x,r))\leq C_s r^s, \quad \forall x\in \mathbb{R}^d,\, \forall r>0
\end{equation}for some $s>0$. Then there exist two subsets $E_1, E_2\subset E$ with positive $\mu$ measure satisfying ${\rm dist}(E_1, E_2)\gtrsim_{s, d} 1$.
\end{lem}

\begin{proof}
Without loss of generality, assume that $E$ is contained in the unit cube. For every integer $k\geq 1$, one can divide $[0,1]^d$ into $2^{dk}$ dyadic sub-cubes. There must be some dyadic sub-cube $Q$ such that $\mu(Q)>0$. If there is another non-adjacent dyadic sub-cube $Q'$ satisfying $\mu(Q')>0$, then we are done. Otherwise, all the measure is contained in $Q$ and its neighbors. It is easy to see that the union of $Q$ and its neighbors are contained in a ball of radius $\left(\sqrt{d}+\frac{\sqrt{d}}{2} \right)2^{-k}$ centered at the center of $Q$. Hence, by the ball condition, one has that
\[
1=\mu(E)\lesssim C_{s, d} 2^{-ks}, 
\]which would lead to a contradiction for sufficiently large $k$.
\end{proof}

In view of this lemma, we will in fact work with two fixed disjoint subsets $E_1, E_2$ below and show that their joint box set $\Box(E_1, E_2):=\{|(x,x)-(y,z)|:\, y\in E_1,\, z\in E_2,\, x\in E\}$ (or the pinned version) has positive Lebesgue measure or contains an interval. This then immediately implies that the same conclusion holds true for the larger set $\Box(E)$.

\subsection*{Further extensions}

The methods we develop in this article are flexible enough to work for many different singular variants of the distance sets.

For instance, suppose $d\geq 1$, and $E,\,F,\,G$ are compact subsets of $\R^d$ with Hausdorff dimensions $s_E,s_F$ and $s_G$, respectively. Define the distance set
\begin{equation}\label{eqn: EFG}
   \Box_{G}(E, F)=\{|(x,x)-(y,z)|\colon (y,z)\in E\times F,\, x\in G\}
\end{equation}
Then one can still ask the question of what conditions on the Hausdorff dimensions $s_E,\,s_F,\,s_G$ one needs to impose so that $\mathcal{L}^{1}(\Box_{G}(E,F))>0$.

Moreover, there is nothing special about the bilinear setting or the choice of the diagonal. Many of our ideas extend readily to the $n$-linear version of the box set
\[
\Box_n(E):=\{|(y^1, y^2, \cdots,y^n)-(x,x,\cdots, x)|:\, x,y^i\in E, \,y^i\neq y^j \text{ if $i\neq j$}\}.
\]

More generally, the diagonal $(x,x,\cdots, x)$ may be replaced by a different singular subset of $E\times \cdots \times E$: $(L_1(x^1,\cdots, x^n),\cdots, L_n(x^1,\cdots, x^n))$, where $L_i$ is a linear form. For example, if one takes $n=3$, $L_1(x^1, x^2, x^3)=x^1$, $L_2(x^1, x^2, x^3)=x^2$, and $L_3(x^1,x^2,x^3)=x^1+x^2$ one gets the following variant of the distance set:
\[
\{|(y^1, y^2, y^3)-(x^1, x^2, x^1+x^2)|:\, y^i, x^i \in E\}.
\]

In addition, we will also be interested in the question of what happens when one replaces the Euclidean metric $|\cdot|$ in $\R^{2d}$ by a more general metric $\|\cdot\|_{K}$ generated by a symmetric bounded convex body $K$ in $\R^{2d}$ with a smooth boundary and everywhere non-vanishing Gaussian curvature. This question has been sought after for the classical Falconer distance problem in many works in the literature (e.g. \cite{IMT2012,GIOW,DIOWZ,HI2005}). Our methods also extend to this general setting.

Lastly, we will describe in Section \ref{sec: finite fields} how our main results can be obtained similarly in the finite fields setting. This initiates the study of such singular variants of the distance problem in this setting and further illustrates the flexibility of our framework.  

\subsection*{Acknowledgements} A.I. is supported in part by NSF DMS-2154232. Y.O. is supported in part by NSF DMS-2055008 and NSF DMS-2142221. The authors would like to thank Jill Pipher for many helpful discussions during the completion of the project.

\vskip.125in 

\section{The sum set approach}\label{sumsetsection}

\vskip.125in

In this section, we introduce our first approach via viewing the box set as a sum set of two distance sets. This enables us to prove Theorem \ref{thm: interval}.

\begin{thm} \label{sumsettheorem} Suppose that the Hausdorff dimension of $E$ is greater than $s_0$, where $s_0$ is the threshold such that the following statement is true:
if $A_1, A_2\subset \mathbb{R}^d$ are two separated compact sets with Hausdorff dimension greater than $s_0$, equipped with the probability measure $\mu_1, \mu_2$ respectively, then there is $A_2'\subset A_2$ with $\mu_2(A_2')>\frac{1}{2}$ such that the Lebesgue measure of the pinned distance set $\{|x-y|: y \in A_1\}$ is positive for $\mu_2$-almost every $x \in A_2'$. 

Then there is $x\in E$ such that $\Box_x(E)$ contains an interval. 
\end{thm}

\vskip.125in 

\begin{proof}To prove Theorem \ref{sumsettheorem}, for every point $x$, observe that we may consider a slightly modified version of $\Box_x(E)$, namely 
$$ \Box_x'(E)=\{{|(y,z)-(x,x)|}^2: y,z \in E,\, y\neq z \}.$$ It is direct to see that $\Box_x(E)$ contains an interval if and only if $\Box'_x(E)$ does.

Similarly, if one further restricts $y,z$ to two disjoint subsets of $E$, chosen according to Lemma \ref{lem: subsets}, one has
\[
\Box_x'(E) \supset \Box_x'(E_1, E_2):=\{|(y,z)-(x,x)|^2:\, y\in E_1, \, z\in E_2\}.
\]In fact, we can choose $E_1, E_2$ such that even better property holds true. More precisely, applying Lemma \ref{lem: subsets} repeatedly, one can obtain three subsets $E_1, E_2, E_3\subset E$ such that each pair of two sets among the three are separated and there is probability measure $\mu_i$ on each $E_i$, $i=1,2,3$, satisfying (\ref{eqn: ball}) for $s>s_0$.

One also defines for a compact subset $A$ of ${\mathbb R}^d$, $d \ge 2$, and any fixed point $u$ 
$$\Delta_u'(A)=\{{|u-v|}^2: v \in A\}.$$ 

We then see for each fixed $x$ that 

$$ \Box'_x(E_1, E_2)=\Delta_x'(E_1)+\Delta_x'(E_2).$$ 

\vskip.125in 

By the classical Steinhaus theorem, $\Delta_x'(E_1)+\Delta_x'(E_2)$ contains an interval as long as each $\Delta_x'(E_i)$ has positive Lebesgue measure. According to our assumption, this is indeed the case for some common $x \in E_3$, hence the proof is complete. 
\end{proof}


The above theorem can be used as a black box to turn pinned Falconer distance results into box set result. More precisely, in order to get Theorem \ref{thm: interval}, it suffices to recall that when ${\rm dim}(E)$ satisfies the given condition as stated in Theorem \ref{thm: interval}, the required pinned distance  result in the hypothesis of Theorem \ref{sumsettheorem} holds true (see \cite{GIOW, DIOWZ} for the even dimensional case and \cite{DGOWWZ, DZ19} for the odd dimensional case). 

\vskip.125in 

\begin{rem} The method above is very flexible. For example, the same conclusion holds if we replace $\Box(E)$ by 
$$ \Box'_n(E)=\{{|(y^1, \dots, y^n)-(x,x, \dots, x)|}^2: x, y^1, \dots, y^n \in E,\, y^i\neq y^j \text{ if }i\neq j\}. $$

Moreover, $(x,x,\dots,x)$ may be replaced by $(L_1(x^1, \dots, x^n), \dots, L^n(x^1, \dots, x^n))$, for a variety of linear forms $L_i$. 
For example, consider
$$\Box'(E)=\{|(y^1,y^2,y^3)-(x^1,x^2,x^1+x^2)|^2\colon x^i,y^j\in E,\, y^i\neq y^j \text{ if }i\neq j\}.$$ Following the same proof strategy as Theorem \ref{sumsettheorem}, it suffices to show that there are mutually separated subsets $E_1, E_2, E_3\subset E$ such that it is possible to find $x^1, x^2\in E$ satisfying $\Delta'_{x^1}(E_1)$, $\Delta'_{x^2}(E_2)$, and $\Delta'_{x^1+x^2}(E_3)$ all have positive Lebesgue measure. 

To get this, one can for example start with finding $x^2$ such that $\mathcal{L}^1(\Delta'_{x^2}(E_2))>0$. Then, fixing $x^2$, we claim that the same pinned distance set results (in \cite{GIOW, DZ19, DIOWZ, DGOWWZ}) hold for pinned point of the form $x+x^2$. Indeed, all these pinned distance results are obtained from showing certain $L^2$ estimate of the form
\[
\int\int |\sigma_t\ast \mu(x)|^2\,dt d\mu(x)<\infty
\]or some variant of it, where $\mu$ denotes the Frostman measure on $E$. It is direct to see that such bound is equivalent to
\[
\int\int |\sigma_t\ast \mu(x+x^2)|^2\,dt d\mu(x)<\infty.
\]Therefore, one can show that for $x^1$ in a subset of positive $\mu$ measure, both $\mathcal{L}^1(\Delta'_{x^1+x^2}(E))$  and $\mathcal{L}^1(\Delta'_{x^1}(E))$ are positive. The proof is then complete. We omit the details.
\end{rem} 

\vskip.125in 

\begin{rem}\label{lengthstwochains}
The sum set approach can be easily adapted to take care of the following family of norms in $\R^{2d}$. For $1\leq  p <\infty$ let $ \|(y,z)\|_{p}=(|y|^{p}+|z|^{p})^{1/p}$ where $|\cdot|$ still denotes the Euclidean distance in $\R^d$. Consider the corresponding pinned box set 
\begin{equation*}
    \Box_{x}^{\|\cdot\|_p}(E)=\{\|(y,z)-(x,x)\|_p\colon y,z\in E, y\neq z\}
\end{equation*}
If $d\geq 2$, and $E\subset \R^d$ with $\text{dim}(E)>\frac{d}{2}+\frac{1}{4}$ if $d$ is even or $\text{dim}(E)>\frac{d}{2}+\frac{1}{4}+\frac{1}{8d-4}$ if $d$ is odd, then there exists $x\in E $ such that $ \Box^{\|\cdot\|_p}_{x}(E)$ contains an interval. In particular, for $p=1$ that implies that the lengths of two chains with vertices in $E$ form a set of nonempty interior in $\R$ if $\text{dim}(E)$ is above those thresholds. Precisely, there is $x\in E$ such that the set 
$$\{|z-x|+|x-y|\colon y,z\in E,\, y\neq z\}$$
contains an interval.

\end{rem} 

\vskip.25in

  \section{The bilinear Sobolev norm bound method}\label{unpinned}

In this section, we prove the first part of Theorem \ref{thm: measure}. The argument is based on a new  Sobolev type estimates for bilinear spherical averaging operators acting on nonnegative functions and an approach relating this to the box set problem (more precisely the positive Lebesgue measure question). This is inspired by \cite{greenleafiosevich}, where Sobolev estimates for the triangle averaging operator acting on nonnegative functions where used to prove positive Lebesgue measure for the set of three point configurations in $\R^2$, when $\text{dim}_{\mathcal{H}}(E)>7/4$. 

Given a set $E\subset \mathbb{R}^d$ with Hausdorff dimension greater than $\frac{d}{2}+\frac{1}{4}$. For $s\in (\frac{d}{2}+\frac{1}{4}, {\rm dim}(E))$, one can find two disjoint subsets $E_1, E_2\subset E$ according to Lemma \ref{lem: subsets}, each of which has positive $s$-dimensional Hausdorff measure. It suffices to focus our study on the box set between $E_1, E_2$, since $\Box (E)\supset \Box_E (E_1, E_2)$ (defined in (\ref{eqn: EFG})).

Let $\mu$  be a Frostman measure on $E$, and $\mu_i$ be a Frostman measure on $E_i$, $i=1,2$, all satisfying the $s$-dimensional ball condition as in (\ref{eqn: ball}). The crucial estimate for this section will be
\begin{equation}\label{mainestimate}
      (\mu\times \mu_1 \times \mu_2)\{(x,y,z)\in E\times E_1\times E_2\colon |(y,z)-(x,x)|\in [t-\eps,t+\eps]\}\leq C\eps 
  \end{equation}
for any $t\sim 1$ where $C>0$ is independent of $\eps>0$ and $t\sim1$.
\begin{thm}\label{linftytheorem}
Assume $d\geq 1$. Let $E\subseteq [0,1]^{d}$ be a compact set. 
\begin{enumerate}
    \item If $\text{dim}_{\mathcal{H}}(E)>d/2+1/4$, then inequality (\ref{mainestimate}) holds.
    \item If $\text{dim}_{\mathcal{H}}(E)>d/2+1/4$, then $\mathcal{L}^{1}(\Box(E))>0$.
\end{enumerate}

\end{thm}

Let $\sigma_r$ denote the normalized surface measure on $S^{2d-1}_r$, the sphere of radius $r$ centered at the origin in $\R^{2d}$. When $r=1$ we oftentimes drop the $r$ and write $\sigma=\sigma_1$. With this normalization one has the bilinear spherical averaging operator
\begin{equation}
    \begin{split}
        A_r(f,g)(x):=&\int_{S^{2d-1}}f(x-ry)g(x-rz)\,d\sigma(y,z)\\
        =&\int_{S^{2d-1}_r} f(x-y)g(x-z)\,d\sigma_r(y,z)\\
        =&\int_{\R^{2d}}\hat{f}(\xi)\hat{g}(\eta)\hat{\sigma}(r\xi,r\eta)e^{2\pi i x\cdot (\xi+\eta)}\,d\xi d\eta.
    \end{split}
\end{equation}

For any $\eps>0$, let $\sigma_r^{\eps}=\sigma_r*\rho_{\eps}$ where $\rho\in C^{\infty}_{0}(B^{2d}(0,1))$ is a non-negative non-increasing radial function satisfying $\int_{\R^{d}\times \R^d}\rho(y,z)\,dydz=1$ and $\rho_{\eps}(y,z):=\eps^{-d}\rho({\eps}^{-1}(y,z))$.

For $\eps, r>0$ we define
\begin{equation}
    A_{r}^{\eps}(f,g)(x):=\int_{\R^d} \int_{\R^d} f(x-y)g(x-z)\sigma_{r}^{\eps}(y,z)\,dydz.
\end{equation}
When $r=1$ we may also write $A_1^{\eps}(f,g)=A^{\eps}(f,g)$.

We first show that part (2) of Theorem \ref{linftytheorem} follows from part (1). In other words,

\begin{lem} Let $E\subseteq \R^d$ be a compact set for which estimate (\ref{mainestimate}) holds. Then, $\mathcal{L}^{1}(\Box(E))>0$.
\end{lem}

 \begin{proof}

Take $t_l,\eps_l>0$ such that $\Box(E)\subseteq \bigcup_{l} I_l$ where $I_l=[t_l-\eps_l,t_l+\eps_l]$ and $\eps_l\leq r$, $r$ small to be chosen later. Since $E$ is compact we can assume $t_l\lesssim 1$, and we just need to show that $\sum_{l}\eps_l\geq c$ for any such cover.
\begin{equation}
\begin{split}
     1=&(\mu\times \mu_1 \times \mu_2)(E\times E_1\times E_2)\\
    \leq &(\mu\times \mu_1\times \mu_2)(\{(x,y,z)\in E\times E_1\times E_2\colon |(y,z)-(x,x)|\in \bigcup_{l\colon t_l\leq r}I_l\}) \\
    +&\sum_{t_l\geq r} (\mu\times \mu_1\times \mu_2)(\{(x,y,z)\in E\times E_1\times E_2\colon |(y,z)-(x,x)|\in I_l\}).
\end{split}
\end{equation}

For the second term where $t_l\sim 1$ we get from (\ref{mainestimate}) that 
$$\sum_{t_l\geq r} (\mu\times \mu_1\times \mu_2)(\{(x,y,z)\in E\times E_1\times E_2\colon |(y,z)-(x,x)|\in I_l\})\leq C\sum_l \eps_l.$$
For the first one we note that since $0\leq t_l,\eps_l\leq r$, $|(y,z)-(x,x)|\leq 2r$ for any points $(x,y,z)$ in that set, so its measure will be bounded by 
\begin{equation}
    \begin{split}
       &\int_{E}(\mu_1\times \mu_2)(\{(y,z)\colon |(y,z)-(x,x)|\leq 2r\})\,d\mu(x)\\
       \leq &\int_{E}(\mu_1\times \mu_2)(\{(y,z)\colon |y-x|\leq 2r\text{ and }|z-x|\leq 2r\})\,d\mu(x)\\
    \leq &\mu_1(B(x,2r))\mu_2(B(x,2r))\mu(E)\lesssim r^{2s}.
    \end{split}
\end{equation}
Choosing $r>0$ sufficiently small one have the desired inequality, say $1/2\leq C\sum_l \eps_l$.
 \end{proof}
In what follows 
$$\|f\|_{L^2_{-\beta}}^2=\int_{\R^d} |\hat{f}(\xi)|^2|\xi|^{-2\beta}\,d\xi.$$

Our first Sobolev norm estimate of the section is the following.

\begin{thm}\label{sobolevestimates}
    Let $A^{\eps}_t$ be as above in $d\geq 1$, where $t\sim 1$. Then if $\beta_1+\beta_2=\frac{2d-1}{2},\,\beta_1,\,\beta_2\geq 0$,
    \begin{equation}
        \|A^{\eps}_t(f,g)\|_{L^{1}}\lesssim  \|f\|_{L^{2}_{-\beta_1}}\|g\|_{L^{2}_{-\beta_2}}
    \end{equation}
    for any $f,g\geq 0$, with constants independent of $\eps>0$ and $t\sim 1$.
\end{thm}

\begin{proof}
For simplicity of notation assume $t=1$. For $t\sim 1 $ one just replaces $\hat{\sigma^{\eps}}(\xi,\eta)$ with $\hat{\sigma^{\eps}}(t\xi,t\eta)$, but all the bounds still hold.

By writing 
$$A^{\eps}(f,g)(x)=\int_{\R^d\times \R^d} \hat{f}(\xi)\hat{g}(\eta)\hat{\sigma^{\eps}}(\xi,\eta)e^{2\pi ix\cdot(\xi+\eta)}\,d\xi d\eta$$
one can check that for a test function $\varphi\in \mathcal{S}(\R^d)$,
\begin{equation}
    \begin{split}
       \int_{\R^d} A^{\eps}(f,g)(x)\hat{\varphi}(x)\,dx=&\int\int \hat{f}(\xi)\hat{g}(\eta)\hat{\sigma^{\eps}}(\xi,\eta)\varphi(\xi+\eta)\,d\eta d\xi\\
       =&\int\left(\int \hat{f}(\xi-\eta)\hat{g}(\eta)\hat{\sigma^{\eps}}(\xi-\eta,\eta)\,d\eta\right) \varphi(\xi)\,d\xi 
    \end{split}
\end{equation}
so that the Fourier transform
$$\mathcal{F}(A^{\eps}(f,g))(\xi)=\int_{\R^d} \hat{f}(\xi-\eta)\hat{g}(\eta)\hat{\sigma^{\eps}}(\xi-\eta,\eta)\,d\eta.$$

Take a radial non-increasing function $\psi\in C^{\infty}_{0}(B^{d}(0,2))$, $\psi\geq 0$, $\psi\equiv 1$ in $\{|x|\leq 1\}$. Since $f,g\geq 0$ and also $\sigma^{\eps}(y,z)\geq 0$ in $\R^{2d}$, it is enough to prove that 
$$\int A^{\eps}(f,g)(x)\psi(x/R)\,dx\leq C\|f\|_{L^{2}_{-\beta_1}}\|g\|_{L^{2}_{-\beta_2}}$$
for any $R\geq 1$ and $\eps>0$.

\begin{equation}
    \begin{split}
        \int A^{\eps}(f,g)(x)\psi(x/R)\,dx=&\int\mathcal{F}(A^{\eps}(f,g))(\xi) R^d\hat{\psi}(R\xi)\,d\xi\\
        =&\int \hat{g}(\eta)\int\hat{f}(\xi-\eta)\hat{\sigma^{\eps}}(\xi-\eta,\eta) R^d\hat{\psi}(R\xi)\,d\xi d\eta\\
        =&\int\int \hat{g}(\eta)\hat{f}(\xi)\hat{\sigma^{\eps}}(\xi,\eta) R^d\hat{\psi}(R(\xi+\eta))\,d\xi d\eta.\\
    \end{split}
\end{equation}
Observe that 
\begin{equation}
    \begin{split}
     |\hat{\sigma^{\eps}}(\xi,\eta)|=|\hat{\sigma}(\xi,\eta)\hat{\rho}(\eps(\xi,\eta))|\leq & \|\hat{\rho}\|_{\infty} |\hat{\sigma}(\xi,\eta)|\\
     \lesssim & \|\rho\|_1\dfrac{1}{(1+|\xi|+|\eta|)^{\frac{2d-1}{2}}}
    \end{split}
\end{equation}
by the well known decay of $\hat{\sigma}_{2d-1}$ obtained from stationary phase arguments.

Using Cauchy-Schwarz inequality and the fact that $\beta_1+\beta_2=\frac{2d-1}{2}$, one gets that
\begin{equation}
    \begin{split}
        \int A^{\eps}(f,g)(x)\psi(x/R)\,dx\lesssim &\int\int \dfrac{|\hat{f}(\xi)||\hat{g}(\eta)|}{(1+|\xi|+|\eta|)^{\beta_1+\beta_2}}R^d|\hat{\psi}(R(\xi+\eta))|\,d\xi d\eta \\
        \lesssim & \left(\int |\hat{f}(\xi)|^2 \left(\int \dfrac{R^d|\hat{\psi}(R(\xi+\eta))|}{(1+|\xi|+|\eta|)^{2\beta_1}}\,d\eta \right)\,d\xi\right)^{1/2} \\
        & \quad\cdot \left(\int |\hat{g}(\eta)|^2 \left(\int \dfrac{R^d|\hat{\psi}(R(\xi+\eta))|}{(1+|\xi|+|\eta|)^{2\beta_2}}\,d\xi \right)\,d\eta\right)^{1/2}.
    \end{split}
\end{equation}

For any $\xi\in \R^d$, and $R\geq 1$, 
$$\int_{\R^d} \dfrac{R^d|\hat{\psi}(R(\xi+\eta))|}{(1+|\xi|+|\eta|)^{2\beta_1}}\,d\eta\leq \dfrac{1}{(1+|\xi|)^{2\beta_1}}\int R^d|\hat{\psi}(R(\xi+\eta))|\,d\eta=\dfrac{1}{(1+|\xi|)^{2\beta_1}}\|\hat{\psi}\|_1$$
and similarly for the symmetric term.

Therefore, 
\begin{equation}
    \begin{split}
        \int A^{\eps}(f,g)(x)\psi(x/R)\,dx\lesssim& \left( \int |\hat{f}(\xi)|^2 \dfrac{1}{(1+|\xi|)^{2\beta_1}} \,d\xi
        \right)^{1/2}\cdot \left(\int |\hat{g}(\eta)|^2  \dfrac{1}{(1+|\eta|)^{2\beta_2}} \,d\eta\right)^{1/2}\\
        \leq& \|f\|_{L^{2}_{-\beta_1}}\|g\|_{L^{2}_{-\beta_2}}.
    \end{split}
\end{equation}

\end{proof}

We will prove Theorem \ref{linftytheorem} using the Sobolev norm estimate obtained in the above. This part of the argument proceeds in a similar fashion as the study of the triangle averaging operator done in \cite{greenleafiosevich}. Our bilinear spherical averaging operator lacks a certain symmetry which was exploited there, bringing in a new challenge that will be described below. To resolve this, we will need to derive the Sobolev norm estimate for another bilinear operator that is closely related to $A_t$. 

\begin{thm}\label{sobolevestimatesforB}
    Let $d\geq 1$, $\varepsilon>0$ and $t\sim 1$. Define
    \begin{equation}
    B_t^{\eps}(f,g)(y):=\int \sigma_t^{\eps}(x-y,x-z)f(z)g(x)\,dzdx.
\end{equation}
Then if $\beta_1+\beta_2=\frac{2d-1}{2},\,\beta_1,\,\beta_2\geq 0$,
    \begin{equation}
\|B^{\eps}_t(f,g)\|_{L^{1}}\lesssim  \|f\|_{L^{2}_{-\beta_1}}\|g\|_{L^{2}_{-\beta_2}}
    \end{equation}
    for any $f,g\geq 0$, with constants independent of $\eps>0$ and $t\sim 1$.
\end{thm}

\begin{proof}
We first claim that 
$$\mathcal{F}(B_t^{\eps}(f,g))(\xi)=\int \hat{f}(-\eta)\hat{g}(\xi+\eta)\hat{\sigma_t^{\eps}}(-\xi,-\eta)\,d\eta.$$
Indeed for $\varphi\in \mathcal{S}(\R^d)$, we can compute formally that
\begin{equation}
    \begin{split}
        \int B_t^{\eps}(f,g)(y)\hat{\varphi}(y)\,dy=&
        \int \int\int \sigma_{t}^{\eps}(x-y,x-z)f(z)g(x)\,dzdx \,\hat{\varphi}(y)dy\\
        =&\int \left(\int\int \sigma_t^{\eps}(x-y,x-z)\hat{\varphi}(y)f(z)\,dydz\right) g(x)\,dx\\
        =&\int \left(\int \int \hat{\sigma_t^{\eps}}(\xi,\eta)\varphi(-\xi)\hat{f}(\eta)e^{2\pi ix\cdot (\xi+\eta)}\,d\xi d\eta\right) g(x)\,dx \\
        =&\int \int\hat{\sigma_t^{\eps}}(\xi,\eta)\varphi(-\xi)\hat{f}(\eta) \check{g}(\xi+\eta)\,d\xi d\eta\\
        =&\int \int \hat{\sigma_t^{\eps}}(-\xi,\eta)\varphi(\xi)\hat{f}(\eta) \check{g}(-\xi+\eta)\,d\xi d\eta\\
        =&\int \varphi(\xi)\left(\int\hat{\sigma_t^{\eps}}(-\xi,\eta)\hat{f}(\eta) \check{g}(-\xi+\eta) \,d\eta\right)\, d\xi\\
         =&\int \varphi(\xi)\left(\int\hat{\sigma_t^{\eps}}(-\xi,-\eta)\hat{f}(-\eta) \hat{g}(\xi+\eta) \,d\eta\right) \,d\xi.
    \end{split}
\end{equation}
    
Hence, let $\psi$ be the same radial bump function as in the proof of Theorem \ref{sobolevestimates}, then one has for all $R\geq 1$ that
\begin{equation}
    \begin{split}
        \int B_t^{\eps}(f,g)(y)\psi(y/R)\,dy=&\int \left(\int \hat{f}(-\eta)\hat{g}(\xi+\eta)\hat{\sigma_t^{\eps}}(-\xi,-\eta)\,d\eta\right) R^d \hat{\psi}(R\xi)\,d\xi\\
        \lesssim & \int \int |\hat{f}(-\eta)||\hat{g}(\xi+\eta)|\dfrac{1}{(1+|\xi|+|\eta|)^{\frac{2d-1}{2}}}R^d |\hat{\psi}(R\xi)|\,d\eta d\xi.\\
    \end{split}
\end{equation}
Since $\beta_1+\beta_2=\dfrac{2d-1}{2}$, the above is 
\begin{equation}
    \begin{split}
        \lesssim & \left(\int |\hat{f}(-\eta)|^2\big(\int \dfrac{1}{(1+|\xi|+|\eta|)^{2\beta_1}} R^d|\hat{\psi}(R\xi)|\,d\xi\big)d\eta\right)^{1/2}\\
    &\quad\cdot\left(\int \int |\hat{g}(\xi+\eta)|^2 \dfrac{1}{(1+|\xi|+|\eta|)^{2\beta_2}} R^d|\hat{\psi}(R\xi)|\,d\xi d\eta\right)^{1/2}\\
        \lesssim & \left(\int |\hat{f}(-\eta)|^2 \dfrac{1}{(1+|\eta|)^{2\beta_1}}\big(\int R^d |\hat{\psi}(R\xi)|\,d\xi\big)\, d\eta\right)^{1/2}\\
        &\quad \cdot\left(\int \big(\int |\hat{g}(\xi+\eta)|^2 \dfrac{1}{(1+|\xi|+|\eta|)^{2\beta_2}}\,d\eta\big) R^d |\hat{\psi}(R\xi)| \,d\xi\right)^{1/2}\\
        \lesssim & \|\hat{\psi}\|_1^{1/2} \|f\|_{L^{2}_{-\beta_1}}\left(\int \big(\int |\hat{g}(\eta)|^2 \dfrac{1}{(1+|\xi|+|\eta-\xi|)^{2\beta_2}}\,d\eta\big) R^d |\hat{\psi}(R\xi)| \,d\xi\right)^{1/2}\\
        \lesssim& \|\hat{\psi}\|_1\|f\|_{L^2_{-\beta_1}}\|g\|_{L^2_{-\beta_2}}.
    \end{split}
\end{equation}
For the last inequality above we used that 
\begin{equation}
    \begin{split}
        \int R^d |\hat{\psi}(R\xi)|\dfrac{1}{(1+|\xi|+|\eta-\xi|)^{2\beta_2}}\,d\xi\lesssim &\dfrac{1}{(1+|\eta|)^{2\beta_2}}\int R^d|\hat{\psi}(R\xi)|\,d\xi\\
        \lesssim & \|\hat{\psi}\|_1\dfrac{1}{(1+|\eta|)^{2\beta_2}}
    \end{split}
\end{equation}
because 
$$1+|\xi|+|\eta-\xi| \gtrsim 1+|\eta|$$
One can easily check the inequality by dividing into the cases $|\xi|\geq |\eta|/2$ or $|\eta|>2|\xi|$. By the Monotone Convergence Theorem, the proof is complete after taking $R\to \infty$. 

\end{proof}

With the above two theorems in tow, we are now ready to prove our main Theorem \ref{linftytheorem}.

\begin{proof}[Proof of Theorem \ref{linftytheorem}]

    Let $t\sim 1$ and $\eps>0$. Observe that
    \begin{equation}
        \begin{split}
              (\mu\times \mu_1 \times \mu_2)&\{(x,y,z)\in E\times E_1\times E_2\colon |(x,x)-(y,z)|\in [t-\eps,t+\eps]\} \\
              =&\int_{E\times E_1\times E_2} \chi_{[t-\eps,t+\eps]}(|(x-y,x-z)|)\,d\mu_1(y)d\mu_2(z)d\mu(x)\\
              \lesssim &\int_{E\times E_1\times E_2} \eps\sigma_t^{\eps}(x-y,x-z)\,d\mu_1(y)d\mu_2(z)d\mu(x).\\
        \end{split}
    \end{equation}
    So we are reduced to showing 
    \begin{equation}\label{eqn: linftygoal}
        \int_{E\times E_1\times E_2} \sigma_t^{\eps}(x-y,x-z)\,d\mu_1(y)d\mu_2(z)d\mu(x)\lesssim 1
    \end{equation}
    uniformly in $\eps>0$ and $t\sim1$.
    
In analogy to what is done in \cite{greenleafiosevich}, we define  
$$\Lambda_t^{\eps}(f_1,f_2,f_3):=\int \int \int \sigma_t^{\eps}(x-y,x-z)f_1(y)f_2(z)f_3(x)\,dydzdx
=\langle A_t^{\eps}(f_1,f_2),f_3\rangle$$
where 
$$A_t^{\eps}(f_1,f_2)(x)=\int\int \sigma_t^{\eps}(x-y,x-z)f_1(y)f_2(z)\,dydz.$$

In \cite{greenleafiosevich}, a similar trilinear form is studied but their trilinear form has also the nice property that
$$\Lambda_{t}^{\eps}(f_1, f_2,f_3)=\Lambda_{\tilde{t}}^{\eps}(f_3,f_2,f_1)$$
where $t=(t_{12},t_{13},t_{23})$ and $\tilde{t}$ is a permutation of that vector. Our trilinear form does not have that symmetry, so we will have to treat the cases $\text{Re}(\alpha)>0$ and $\text{Re}(\alpha)<0 $ separately in the argument below.

For $\delta>0$, let $\hat{\mu^{\delta}}(\xi)=\hat{\mu}(\xi)\hat{\rho}(\delta\xi)$, where $\rho\in C^{\infty}_{0}(B^{d}(0,1))$ non-negative radial function with $\int \rho(y) dy=1$

Define for any probability measure $\nu$ and $\alpha\in \C,$
$$\nu_{\alpha}(x)=\dfrac{2^{-\alpha}\pi^{-d/2}}{\Gamma(\alpha/2)}(\nu*|\cdot|^{-d+\alpha})(x)$$
This is initially well defined for $0<\text{Re}(\alpha)<d$, and it is extended to the complex plane by analytic continuation so that 
$$\hat{\nu}_{\alpha}(\xi)=\dfrac{(2\pi)^{-\alpha}}{\Gamma(\frac{d-\alpha}{2})}|\xi|^{-\alpha}\hat{\nu}(\xi)$$(see \cite{Grafakosbook} for the choice of constants).

Consider the complex function
$$F(\alpha)=\Lambda_{t}^{\eps}(\mu_{1,-\alpha}^{\delta},\mu_2^{\delta},\mu_{\alpha}^{\delta}),\quad\alpha\in \C.$$

First, let us estimate $|F(\alpha)|$ for $\text{Re}(\alpha)>0$.

If $\text{dim}_{\mathcal{H}}(E)>d-\text{Re}(\alpha)$, then $\|\mu_{\alpha}^{\delta}\|_{\infty}\lesssim 1$. Indeed, from the same argument as in \cite{greenleafiosevich}
\begin{equation}\label{independentofdelta}
\begin{split}
     \mu_{\alpha}^{\delta}(x)\lesssim \int |x-y|^{-d+\text{Re}(\alpha)}\,d\mu^{\delta}(y)\lesssim & \sum_{m\geq 0} 2^{m(d-\text{Re}(\alpha))}\int_{|x-y|\sim 2^{-m}}d\mu^{\delta}(y)\\
     \lesssim & \sum_{m\geq 0} 2^{m(d-\text{Re}(\alpha))}2^{-ms}\lesssim 1
\end{split}
\end{equation}
for $d-\text{Re}(\alpha)<s<\text{dim}_{\mathcal{H}}(E)$, independently of $\delta>0$. Combining that with the bounds in Theorem \ref{sobolevestimates}, one has
\begin{equation}
    |F(\alpha)|\leq \|A_t^{\eps}(\mu_{1,-\alpha}^{\delta},\mu_2^{\delta})\|_1\|\mu_{\alpha}^{\delta}\|_{\infty}\lesssim \|\mu_{1,-\alpha}^{\delta}\|_{L^{2}_{-\beta_1}}\|\mu_2^{\delta}\|_{L^{2}_{-\beta_2}}
\end{equation}
for all $\beta_1,\beta_2\geq 0$, with $\beta_1+\beta_2=\dfrac{2d-1}{2}$.

By Plancherel's Theorem and using that $|\hat{\rho}(\delta\xi)|\leq \|\hat{\rho}\|_{\infty}\leq \|\rho\|_1<\infty$,
\begin{equation}
    \begin{split}
    \|\mu_{1,-\alpha}^{\delta}\|_{L^{2}_{-\beta_1}}\|\mu_2^{\delta}\|_{L^{2}_{-\beta_2}}\lesssim &\||\xi|^{\alpha-\beta_1}\hat{\mu}_1(\xi)\hat{\rho}(\delta\xi)\|_{2}\||\xi|^{-\beta_2}\hat{\mu}_2(\xi)\hat{\rho}(\delta \xi)\|_2\\
    \lesssim & \||\xi|^{\text{Re}(\alpha)-\beta_1}\hat{\mu}_1(\xi)\|_{2}\||\xi|^{-\beta_2}\hat{\mu}_2(\xi)\|_2=\||\xi|^{-\beta_2}\hat{\mu}_1(\xi)\|_2\||\xi|^{-\beta_2}\hat{\mu}_2(\xi)\|_2\\
    \end{split}
\end{equation}
if we choose $\text{Re}(\alpha),\beta_1,\beta_2$ such that $\text{Re}(\alpha)-\beta_1=-\beta_2$.

Finally, one has
\begin{equation}
    \begin{split}
      \||\xi|^{-\beta_2}\hat{\mu}_1(\xi)\|_2=&\left(\int \dfrac{|\hat{\mu}_1(\xi)|^{2}}{|\xi|^{2\beta_2}}\,d\xi\right)^{1/2}  \\
      \sim&\left(\int \dfrac{1}{|x-y|^{d-2\beta_2}}\,d\mu_1(x)d\mu_1(y)\right)^{1/2}\lesssim 1
    \end{split}
\end{equation}and similarly $\||\xi|^{-\beta_2}\hat{\mu}_2(\xi)\|_2\lesssim 1$, as long as $\text{dim}_{\mathcal{H}}(E)>d-2\beta_2$.
Collecting all the conditions we want to be satisfied, we have

\begin{enumerate}
    \item $ \beta_1,\beta_2\geq 0\text{ and }\beta_1+\beta_2=\dfrac{2d-1}{2}$;
    \item $\beta_2=\beta_1-\text{Re}(\alpha)$;
    \item $d-\text{Re}(\alpha)<\text{dim}_{\mathcal{H}}(E)$ ;
    \item   $ d-2\beta_2<\text{dim}_{\mathcal{H}}(E)$.
\end{enumerate}
Combining (1) and (2), we get $2\beta_2=\frac{2d-1}{2}-\text{Re}(\alpha)$, so (3) and (4) become 
\begin{equation}
    \begin{cases}
    d-\text{Re}(\alpha)<\text{dim}_{\mathcal{H}}(E)\\
    \text{Re}(\alpha)+1/2<\text{dim}_{\mathcal{H}}(E).
    \end{cases}
\end{equation}
Hence for any $\alpha\in \C$ such that $d-\text{Re}(\alpha)=\text{Re}(\alpha)+1/2$, that is, $\text{Re}(\alpha)=\dfrac{2d-1}{4}$, this gives $\beta_2=\dfrac{2d-1}{8}$ and $\beta_1=\dfrac{3}{8}(2d-1)$, and one gets that for $\text{dim}_{\mathcal{H}}(E)>\dfrac{2d+1}{4}=\dfrac{d}{2}+\dfrac{1}{4}$, $|F(\alpha)|\lesssim 1$.

Next we want to show that if $\text{Re}(\alpha)=-\frac{(2d-1)}{4}$ then for $\text{dim}_{\mathcal{H}}(E)>\frac{d}{2}+\frac{1}{4}$, $|F(\alpha)|\lesssim 1$.

To see this, let $\text{Re}(\alpha)<0$. We can write
\begin{equation}
\begin{split}
    F(\alpha)=&\Lambda^\varepsilon_{t}(\mu_{1,-\alpha}^{\delta},\mu_2^{\delta},\mu_{\alpha}^{\delta})\\
    =&\int \sigma_t^{\eps}(x-y,x-z)\mu_{1,-\alpha}^{\delta}(y)\mu_2^{\delta}(z)\mu_{\alpha}(x)\,dydzdx\\
    =&\int\left( \int\int \sigma_t^{\eps}(x-y,x-z)\mu_2^{\delta}(z)\mu_{\alpha}^{\delta}(x) \,dz dx\right) \mu_{1,-\alpha}^{\delta}(y) \,dy\\
    =&\int B_t^{\eps}(\mu_2^{\delta},\mu_{\alpha}^{\delta})(y)\mu_{1,-\alpha}^{\delta}(y) \, dy
\end{split}
\end{equation}
where
\[
B_t^{\eps}(f,g)(y)=\int \sigma_t^{\eps}(x-y,x-z)f(z)g(x)\,dzdx.
\]

With the same argument as before, the fact that $|F(\alpha)|\lesssim 1$ for $\text{Re}(\alpha)=-\frac{2d-1}{4}$ will follow immediately from Theorem \ref{sobolevestimatesforB}. 

Finally, combining the bounds $|F(\alpha)|\lesssim 1$, independently of $\delta$, in both cases above with the Three lines lemma (\cite{Hirschman}), one concludes that $|F(0)|\lesssim 1$. Letting $\delta\to 0$, one has that this implies precisely the desired estimate (\ref{eqn: linftygoal}). The proof is complete.

\end{proof}

\begin{rem}\label{rem: convex body}
    The arguments in Section \ref{unpinned} still hold if we replace $S^{2d-1}$ with some more general surfaces of $\R^{2d}$. Observe that what was crucial about $\hat{\sigma}$ in our proof was the decay bound
    $$|\hat{\sigma}(\xi,\eta)|\lesssim \dfrac{1}{(1+|\xi|+|\eta|)^{\frac{2d-1}{2}}}.$$
    If $\|\cdot\|_{K}$ is a norm generated by a symmetric bounded convex body $K$ in $\R^{2d}$ with a smooth boundary and everywhere non-vanishing Gaussian curvature, consider 
    $$\Box^{\|\cdot\|_{K}}(E)=\{\|(x,x)-(y,z)\|_{K}\colon x,y,z\in E,\, y\neq z\}.$$
    Since the Fourier transform of the surface measure in the boundary of $K$ satisfies the same decay bound (see \cite{steinhabook} or \cite{booksogge}, for example), we still get the same conclusion, that is, $\mathcal{L}^{1}(\Box^{\|\cdot\|_{K}}(E))>0 $ if $\text{dim}_{\mathcal{H}}(E)>d/2+1/4$, and $d\geq 1$.
\end{rem}  

\section{An alternative proof for the case $d=1$}\label{sec: product set}

In this section, we exploit a different approach, along the lines of \cite{GILP2015} and \cite{Liu19}. We first recover the result in the previous section for $d=1$ and then upgrade it to a pinned version. Namely, 

\begin{thm} \label{twosetsthm} 
Let $E\subseteq \R$ be a compact set. If $\text{dim}_{\mathcal{H}}(E)>1/2+1/4=3/4$, then 
    $\mathcal{L}^{1}(\Box(E))>0$. Moreover, there is a point $x\in E$ such that $\mathcal{L}^1(\Box_x(E))>0$.
\end{thm}

To prove Theorem \ref{twosetsthm}, we first realize $\Box(E)$ as a set of distances between two different subsets of $\R^2$. Given $A,B\subset \R^2$ compact subsets, define 
\begin{equation}
    \Delta(A,B)=\{|a-b|\colon a\in A,b\in B\}.
\end{equation}

For $E\subset \R$, choose two disjoint subsets $E_1, E_2$ of it as in Lemma \ref{lem: subsets}, and let 
$$\mathcal{D}_{E\times E}:=\{(y,y)\colon y\in E\}$$ denote the diagonal of $E\times E$.

Then 
\begin{equation}
\begin{split}
    \Box(E)\supset \{|(x,x)-(y,z)|\colon x\in E,\, y\in E_1, \, z\in E_2\}=\Delta(A,B)
\end{split}
\end{equation}
for $A=E_1\times E_2\subseteq \R^2$ and $B=\mathcal{D}_{E\times E}\subseteq \R^2$. The problem is thus reduced to showing that $\mathcal{L}^1(\Delta(A,B))>0$.

In the following, we first prove the unpinned version of the result by describing a group action method inspired from \cite{GILP2015}. Even though this result will be superseded by the pinned result later, we find this argument interesting in its own right and may inspire some future applications. 

\subsection{An argument for the unpinned result}

Consider the induced distance measure $\nu_{A,B}$ in $\Delta(A,B)$, namely
$$\int \varphi(t)\,d\nu_{A,B}(t)=\int \varphi(|a-b|)\,d\mu_A(a)d\mu_B(b),\quad \forall\varphi\in \mathcal{S}(\R),$$
where $\mu_A,\mu_B$ are Frostman measures on $A$ and $B$, respectively.

One has 
\begin{equation}
    \begin{split}
        \nu_{A,B}^{\eps}(t):=&\frac{1}{\eps}\int \chi_{[t-\eps,t+\eps]}(t)\,d\nu_{A,B}(t)\\
        =&\frac{1}{\eps} \int_{A\times B}\chi_{[t-\eps,t+\eps]}(|a-b|)\,d\mu_A(a)d\mu_B(b),
    \end{split}
\end{equation}
so
\begin{equation}
    \begin{split}
        \int \nu^{\eps}_{A,B}(t)^2 dt=&\frac{1}{\eps^2} \int \chi_{[t-\eps,t+\eps]}(|a-b|) \chi_{[t-\eps,t+\eps]}(|a'-b'|) \,d\mu_A(a)d\mu_B(b) d\mu_A(a')d\mu_B(b')dt \\
        \lesssim &\frac{1}{\eps^2}\int \chi_{||a-b|-|a'-b'||\leq 2\eps} \left(\int \chi_{||a-b|-t|\leq \eps}(t)\,dt\right) \,d\mu_A(a)d\mu_B(b) d\mu_A(a')d\mu_B(b')\\
        \lesssim & \frac{1}{\eps}\int \chi_{||a-b|-|a'-b'||\leq 2\eps} \,d\mu_A(a)d\mu_B(b) d\mu_A(a')d\mu_B(b').
    \end{split}
\end{equation}

Cover the circle $S^1$ with arcs $C_j\subset O(2)$ of length comparable to $\eps$ (so wee need $\sim 1/\eps$ many of them). For any choice of $\theta_j\in C_j$, $1\leq j\leq 1/\eps$, one has
$$\{(a,b,a',b')\colon ||a-b|-|a'-b'||\leq 2\eps\}\subseteq \bigcup_{1\leq j\leq 1/\eps}\{(a,b,a',b')\colon |(a'-b')-\theta_j(a-b)|< 3\eps\}.$$
Hence,
\begin{equation}
    \begin{split}
        \int \nu_{A,B}^{\eps}(t)^2\,dt\lesssim& \frac{1}{\eps}\sum_{j} \frac{1}{\eps}\int_{C_j} (\mu_A\times \mu_B\times \mu_A\times \mu_B)
\{(a,b,a',b')\colon |(a'-b')-\theta(a-b)|<4 \eps\}d\theta\\
\lesssim & \frac{1}{\eps^2}\int_{O(2)}  (\mu_A\times \mu_B\times \mu_A\times \mu_B)
\{(a,b,a',b')\colon |(a'-b')-\theta(a-b)|<4 \eps\}d\theta.
\end{split}
\end{equation}

Also define measures $\lambda_{\theta}^{A},\,\lambda_{\theta}^{B} $ on $A-\theta A$ and $B-\theta B$, respectively, by
\begin{equation}
    \begin{split}
        \int \varphi(z)\,d\lambda_{\theta}^{A}(z)=& \int\varphi(a-\theta a')\,d\mu_{A}(a)d\mu_{A}(a'),\\
        \int \varphi(z)\,d\lambda_{\theta}^{B}(z)=&\int \varphi(b-\theta b')\,d\mu_B(b)d\mu_B(b').
    \end{split}
\end{equation}

Let $\psi^{\eps}(z)=\eps^{-2}\psi(z/\eps),\,z\in \R^2$ be an approximation of the identity in $\R^2$, where $\psi\in C^{\infty}_{0}(\R^2)$, $\psi\equiv 1$ in $|z|\leq 1$ and define
$$\lambda^{A,\eps}_{\theta}:=\lambda_{\theta}^{A}*\psi^{\eps}.$$

We claim that 
\begin{equation}\label{ABineqwitheps}
    \int_{O(2)}\int_{\R^2} \lambda_{\theta}^{A,\eps}(z)\lambda_{\theta}^{B,\eps}(z)\,dzd\theta \gtrsim \int \nu_{A,B}^{\eps/8}(t)^2\, dt.
\end{equation}

To prove our claim notice that 
\begin{equation}
    \begin{split}
        \lambda_{\theta}^{A,\eps}(z)=&\langle \lambda_{\theta}^{A}, \psi^{\eps}(\cdot-z)\rangle\\
        =&\int \psi^{\eps}(a'-\theta a-z)d\mu_A(a)d\mu_{A}(a')\\
        \gtrsim& \int\frac{1}{\eps^2}\chi_{B(0,\eps)}(a'-\theta a-z)d\mu_A(a)d\mu_{A}(a').
    \end{split}
\end{equation}Therefore,
\begin{equation*}
    \begin{split}
&\int_{O(2)}\int_{\R^2}\lambda_{\theta}^{A,\eps}(z)\lambda_{\theta}^{B,\eps}(z)\,dz d\theta\\
        \gtrsim &\int(\int_{\R^2} \eps^{-2} \chi_{B(0,\eps)}(a'-\theta a-z)\eps^{-2} \chi_{B(0,\eps)}(b'-\theta b-z)\,dz) \,d(\mu_A^2\times \mu_B^2)(a,a',b,b')d\theta\\
        \gtrsim & \int \eps^{-4} \chi_{B(0,\frac{\eps}{2})}((a'-\theta a)-(b'-\theta b))\left(\int_{\R^2} \chi_{B(0,\frac{\eps}{2})}(a'-\theta a -z )\,dz\right)\,d(\mu_A^2\times \mu_B^2)(a,a',b,b')d\theta\\
        \gtrsim &\int_{O(2)}\int_{A^2\times B^2} \eps^{-2} \chi_{|(a'-b')-\theta (a-b)|\leq \eps/2}\,d(\mu_A^2\times \mu_B^2)(a,a',b,b')d\theta\\
        \gtrsim & \int \nu_{A,B}^{\eps'}(t)^2\,dt.
    \end{split}
\end{equation*}
for say $\eps'=\eps/8$.

Letting $\eps \rightarrow 0$ in inequality (\ref{ABineqwitheps}), we get the following conclusion. 

\begin{prop} Let $\lambda_{\theta}^{A}$, $\lambda_{\theta}^{B}$ and $\nu_{A,B}$ be as above. Then,
    \begin{equation}
        \int_{\R} \nu_{A,B}(t)^2 \,dt\lesssim \int_{O(2)}\int_{\R^2} \lambda_{\theta}^{A}(z)\lambda_{\theta}^{B}(z)\,dz d\theta.
    \end{equation} 
\end{prop}
This means that if the right hand side of this inequality is finite then the measure $\nu_{A,B}$ has an $L^2$ density which we still denote by $\nu_{A,B}$, which implies by standard Cauchy-Schwarz argument that $\mathcal{L}^1(\Delta(A,B))>0$.

\begin{rem}
    The proposition above generalizes for $A,\, B\subset \R^d$ and $d\geq 2$, by replacing $O(2)$ by $O(d)$. The key ideas can be found in \cite{GILP2015}. We omit the details.
\end{rem}

Let us now apply the proposition above to our case of interest.

\begin{prop}\label{propAB}
    Let $A, B\in \mathbb{R}^2$ be compact sets with Hausdorff dimension ${\rm dim}(A)>\frac{3}{2}$ and ${\rm dim}(B)>\frac{1}{2}$. Then, 
    $$\int_{O(2)}\int_{\R^2} \lambda_{\theta}^{A}(z)\lambda_{\theta}^{B}(z)\,dz d\theta\lesssim 1.$$
    As a consequence, $\nu_{A,B}\in L^{2}(\R)$ so $\mathcal{L}^{1}(\Delta(A,B))>0$.
\end{prop}

It is easy to see that the first part of Theorem \ref{twosetsthm} follows immediately from this proposition. Indeed, recall from the beginning of the section that we have chosen $A=E_1\times E_2$ and $B=\mathcal{D}_{E\times E}$. If ${\rm dim}(E)>\frac{3}{4}$, then ${\rm dim}(A)\geq {\rm dim}(E_1)+{\rm dim}(E_2)>\frac{3}{4}+\frac{3}{4}=\frac{3}{2}$ and ${\rm dim}(B)={\rm dim}(E)>\frac{1}{2}$.

\begin{proof}[Proof of Proposition \ref{propAB}]
From the definition of $\lambda_{\theta}^{A}$ one has
$$\hat{\lambda}_{\theta}^{A}(\xi)=\hat{\mu}_A(\xi)\overline{\hat{\mu}_{A}(\theta^{-1}\xi)}.$$
So from H\"older's inequality followed by change of variables we get
\begin{equation}
    \begin{split}
        \int_{O(2)} \int_{\R^2}  \lambda_{\theta}^{A}(z)&\lambda_{\theta}^{B}(z)\,dz d\theta=\int_{O(2)}\int_{\R^2} \hat{\lambda}_{\theta}^{A}(\xi)\overline{\hat{\lambda}_{\theta}^{B}(\xi)}\,d\xi d\theta\\
        =&\int_{O(2)}\int_{\R^2}\hat{\mu}_{A}(\xi)\overline{\hat{\mu}_{A}(\theta^{-1}\xi)}\overline{\hat{\mu}_{B}(\xi)}\hat{\mu}_{B}(\theta^{-1}\xi)\,d\xi d\theta\\
        \lesssim& \left(\int_{O(2)} \int_{\R^2}|\hat{\mu}_{A}(\xi)|^2|\hat{\mu}_{B}(\theta^{-1}\xi)|^2 \,d\xi d\theta \right)^{1/2}\\
        &\cdot\left(\int_{O(2)} \int_{\R^2}|\hat{\mu}_{B}(\xi)|^2|\hat{\mu}_{A}(\theta^{-1}\xi)|^2 \,d\xi d\theta \right)^{1/2}\\
        =&\int_{O(2)} \int_{\R^2}|\hat{\mu}_{A}(\xi)|^2|\hat{\mu}_{B}(\theta\xi)|^2 \,d\xi d\theta.
    \end{split}
\end{equation}

\begin{lem}\label{lemmaforB}
    For $B\subset \R^2$ with $\text{dim}_{\mathcal{H}}(B)>1/2$, one has for all $\xi\in \R^2\backslash\{ 0\}$
    $$\int_{O(2)}|\hat{\mu}_{B}(\theta \xi)|^2\,d\theta\lesssim \dfrac{1}{|\xi|^{1/2}}.$$

\end{lem}
\begin{proof}[Proof of Lemma \ref{lemmaforB}]
It is equivalent to show that for $r>0$ 
$$\int_{S^1} |\hat{\mu}_B(r\omega )|^2 \,d\sigma(\omega)\lesssim r^{-1/2}.$$
The left hand side of this inequality is equal to 
\begin{equation}
    \begin{split}
       & \int_{S^1} \int \int e^{2\pi i(x-y)\cdot r\omega} \,d\mu_B(x)d\mu_B(y)d\sigma(\omega)\\
        =&\int \int \hat{\sigma}(r(x-y))\,d\mu_B(x)d\mu_B(y)\\
    \lesssim&\,r^{-1/2}\int |x-y|^{-1/2}\,d\mu_B(x)d\mu_B(y)\lesssim r^{-1/2}
    \end{split}
\end{equation}
    if $\text{dim}_{\mathcal{H}}(B)>1/2$.
\end{proof}

Plugging the estimate in the lemma above for $B$ into the previous inequality we get 
\begin{equation}
\begin{split}
    \int\int \lambda_{\theta}^{A}(z) \lambda_{\theta}^{B}(z)\,dz d\theta\lesssim & \int_{\R^2} \int_{O(2)} |\hat{\mu}_{A}(\xi)|^2|\hat{\mu}_{B}(\theta\xi)|^{2}\,d\theta d\xi\\
    \lesssim &\int_{\R^2} |\hat{\mu}_{A}(\xi)|^2 \dfrac{1}{|\xi|^{1/2}}\,d\xi\\
    \lesssim & \int \dfrac{1}{|a-a'|^{3/2}}\,d\mu_{A}(a)d\mu_{A}(a')<\infty
\end{split}
\end{equation}
if $\text{dim}_{\mathcal{H}}(A)>3/2$. The proof is complete.
\end{proof}

\subsection{The pinned result}

We now turn to prove the second part of Theorem \ref{twosetsthm}, in other words the pinned version of the result. We first recall a result by Liu \cite{Liu19} where the pinned Falconer distance problem is studied via a similar group action method. The following is a rephrased version of Theorem 1.7 in \cite{Liu19}.

\begin{thm}\label{Liu}\cite{Liu19}
Let $d\geq 2$ and $A, B\subset \mathbb{R}^d$ be compact sets. Let $\mu_B$ be the Frostman measure on $B$ satisfying
\begin{equation}\label{eqn: Liu1}
\mu_B(B(x,r))\lesssim r^{s_B},\quad \forall x\in \mathbb{R}^d,\, \forall r>0,
\end{equation}and
\begin{equation}\label{eqn: Liu2}
\int_{S^{d-1}} |\hat\mu_B(R\omega)|^2\,d\sigma(\omega)\lesssim_\epsilon R^{-\beta(s_B)+\epsilon},\quad \forall \epsilon>0.
\end{equation}Suppose further that ${\rm dim}(A)+\beta(s_B)>d$. Then, for $\mu_B$-almost every $x\in B$, there holds $\mathcal{L}^1(\Delta_x(A)):=\mathcal{L}^1(\{|x-y|:\, y\in A\})>0$. In particular, $\mathcal{L}^1(\Delta(A, B))>0$.
\end{thm}

Back to our problem, where $A=E_1\times E_2\subset \R^2$ and $B=\mathcal{D}_{E\times E}\subset \R^2$. For all $s_B<{\rm dim}(B)={\rm dim}(E)$, (\ref{eqn: Liu1}) holds true. If one assumes that $s_B>\frac{1}{2}$, then (\ref{eqn: Liu2}) holds true with $\beta(s_B)=\frac{1}{2}$ according to Lemma \ref{lemmaforB}. Therefore, to conclude the pinned result in Theorem \ref{twosetsthm}, it suffices to have the condition ${\rm dim}(A)+\beta(s_B)>2$, which can be implied by ${\rm dim}(E)>\frac{3}{4}$. The proof of Theorem \ref{twosetsthm} and hence Theorem \ref{thm: measure} is complete.

\subsection{Some further remarks}

     Both strategies described above can take care of more general box sets. Given $E,\,F,\,G\subset \R$ with $\text{dim}_{\mathcal{H}}(G)>\frac{1}{2}$ and $\text{dim}_{\mathcal{H}}(E\times F)>3/2$. Define their associated box set as $\Box_G(E, F):=\{|(y,z)-(g,g)|:\, y\in E,\, z\in F, g\in G\}$ and similarly its pinned version $\Box_{g}(E, F)$. Then by taking $A=E\times F$ and $B=\mathcal{D}_G=\{(g,g)\colon g\in G\}$ in the argument above, one gets the following.

\begin{cor}\label{n2d1}
    Let $E,\,F,\,G\subset \R$ with 
    $$\text{dim}_{\mathcal{H}}(G)>\frac{1}{2} \text{ and }\text{dim}_{\mathcal{H}}(E\times F)>\frac{3}{2}.$$
    Then $\mathcal{L}^{1}(\Box_{G}(E, F))>0$ and there exists $g\in G$ such that $\mathcal{L}^1(\Box_g(E, F))>0$.
\end{cor}

Furthermore, the same proof strategies also work in all dimensions $d\geq 1$. However, they will only yield a dimensional threshold that is much worse than the one obtained by the bilinear method in Section \ref{unpinned} if $d\geq 2$. More generally, for $E\subset \mathbb{R}^d$, $d\geq 1$, one can in fact consider multilinear versions of $\Box(E)$. For $n\geq 2$, let
    \begin{equation}
        \Box_n(E)=\{|(y^1,y^2,\dots, y^n)-(x,x,\dots,x)|\colon x,y^i\in E,1\leq i\leq n,\, y^i\neq y^j \text{ if } i\neq j\},
    \end{equation}
    so when $n=2$ one just recovers $\Box(E)$. Similarly, one can define the pinned version $\Box_{n,x}(E)$.

    Adapting the proof of Theorem \ref{twosetsthm}, we have the following corollary.

\begin{cor}
        Let $n\geq 2$ and $d\geq 1$. Let $E\subset \R^d$ satisfy 
        \begin{equation}
            \text{dim}_{\mathcal{H}}(E)>\begin{cases}
            3/4,\,\text{if }(n,d)=(2,1),\\
                d\left(\frac{n}{n+1}\right),\,\text{if }(n,d)\neq (2,1).
            \end{cases}
        \end{equation}
        Then, there exists $x\in E$ such that $\mathcal{L}^1(\Box_{n,x}(E))>0$. 
    \end{cor}

\begin{proof}
The case $(n,d)=(2,1)$ is already included in Corollary \ref{n2d1}, so let us assume $(n,d)\neq (2,1)$.
We only give a sketch as the argument is very similar to the proof of Theorem \ref{twosetsthm}. 

Let $A=E^n\subset \R^{nd}$ and $B=\mathcal{D}_E:=\{(x,x,\dots,x)\colon x\in E\}\subset \R^{nd}$. Let $\beta(s_B)$ be the decay rate in (\ref{eqn: Liu2}). It is a classical result of Mattila \cite{Mattila87} that (\ref{eqn: Liu2}) holds true with
\[
\beta(s)=\begin{cases} s,& s \in (0, \frac{nd-1}{2}],\\ \frac{nd-1}{2}, & s \in [\frac{nd-1}{2}, \frac{nd}{2}]. \end{cases}
\]Therefore, the desired result follows easily by applying Theorem \ref{Liu} in the suitable range of dimensions. For example, if $n\geq 3$, then there holds $\frac{nd-1}{2}\geq d\geq {\rm dim}(B)>s_B$, hence, $\beta(s_B)=s_B$. Since ${\rm dim}(A)\geq n{\rm dim}(E)$, one can take $s_B={\rm dim}(E)-\epsilon$ and concludes from the condition ${\rm dim}(E)> \frac{nd}{n+1}$ that ${\rm dim}(A)+\beta(s_B)>nd$. The rest of the cases can be taken care of similarly and are left to the interested reader.
\end{proof}

\begin{rem}
Following the same proof strategy in \cite{GIOW} for the two dimensional Falconer distance problem, one can in fact show another result concerning the joint distance set $\Delta(A, B)$ for separated compact sets $A, B\subset \mathbb{R}^2$. More precisely, If ${\rm dim}(A)>1$, ${\rm dim}(B)>1$ and $3{\rm dim}(A)+{\rm dim}(B)>5$, then there exists $x\in B$ such that $\mathcal{L}^1(\Delta_x(A))>0$.

Unfortunately, this estimate cannot be applied to our box set problem, since in our case $B=\mathcal{D}_{E\times E}$ has Hausdorff dimension less than $1$. A similar obstruction remains in higher dimensions for applications to the study of the box set problem.
\end{rem}

\vskip.125in 

\section{The projection method} \label{sec: projection}

\vskip.125in 

In this section, we introduce yet another approach that allows us to handle a variety of problems. A systematic study will be undertaken in the sequel, but we introduce the notion here and prove some sample results.

\subsection{Positive measure for set of lengths of non-degenerate chains}
For a compact set $E\subset \R^d,\, d\geq 2$ and $k\geq 1$, the set of non-degenerate $k$-chains generated by $E$ is given by
$$ C^k(E)=\{(|x_0-x_1|, |x_1-x_2|, \dots, |x_{k-1}-x_{k}|): x_j \in E,\,x_0,x_1\dots,x_{k} \text{ distinct}  \}\subset \R^k.$$

For any $x\in E$, one can also consider a pinned version of $C^{k}(E)$ comprised of non-degenerate $k$-chains generated by $E$ with starting place at $x$, namely
$$ C^k_{x}(E)=\{(|x-x_1|, |x_1-x_2|, \dots, |x_{k-1}-x_{k}|): x_j \in E,\,x_1,x_2\dots,x_{k} \text{ distinct}  \}.$$

It was proved by the third listed author of this paper and Krystal Taylor (\cite{OT2020}) that if the Hausdorff dimension of $E \subset {\mathbb R}^d$, $d \ge 2$ is greater than $s_0$, where $s_0$ is the best-known threshold for the pinned Falconer distance problem, then there exists $x\in E$ such that 
$$ {\mathcal L}^k(C^k_x(E))>0\,\text{ for all }k\geq 1,$$
where $\mathcal{L}^k$ stands for the Lebesgue measure in $\R^k$.
For $k\geq 2$, let us now look at the set of lengths attained by non-degenerate $k$-chains in $E$
$$L^{k}(E):=\{t_1+t_2+\dots+t_k\colon (t_1,\dots, t_k)\in C^k(E)\}\subset \R$$
and its pinned version
$$L^{k}_x(E)=\{t_1+t_2+\dots+t_k\colon (t_1,\dots, t_k)\in C^k_x(E)\},\, \text{ where }x\in E.$$

For simplicity, take the case $k=2$. Consider a map 
$$ (t_1,t_2) \mapsto (t_1+t_2, t_1-t_2)$$ and note that the Jacobian is bounded from above and below. It follows that the image of $C^2_x(E)$ under this map still has a positive two-dimensional Lebesgue measure, and hence the projection of this image onto the first variable (or the second, for that matter) has a positive one-dimensional Lebesgue measure. That is exactly to say that $L^{2}_x(E)$ has positive one-dimensional Lebesgue measure when dim$(E)>s_0$. 
More generally, for any $k\geq 2$, by picking a suitable linear transformation $T^k:\R^{k}\rightarrow \R^k$ with $(\pi_1\circ T)(t_1,t_2,\dots ,t_k)=\sum_{i=1}^{k} t_i$ and Jacobian comparable to 1, we conclude that $\mathcal{L}^1(L_x^{k}(E))>0$ when $\text{dim}(E)>s_0$.

\begin{rem}
    One can easily recover from this projection argument part of Theorem \ref{thm: interval}, more precisely the fact that $\Box(E)$ has positive measure for $d\geq 2$ and $E\subset \R^d$ compact with dimension larger than $s_0$. Indeed, it is straightforward to see that the argument in \cite{OT2020} still works if one replaces the set of chains $C_x^k(E)$ by 
    \[
    \tilde{C}_x^k(E):=\{(|x-x_1|^2,|x_1-x_2|^2,\ldots, |x_{k-1}-x_k|^2):\, x_j\in E,\,x_1,x_2\dots,x_{k} \text{ distinct}\}.
    \]Then, replacing its lengths set $L_x^k(E)$ by
    \[
    \tilde{L}_x^k(E):=\{t_1+t_2+\cdots + t_k:\, (t_1,\ldots, t_k)\in \tilde{C}_x^k(E)\},
    \]when ${\rm dim}(E)>s_0$, one has that $\mathcal{L}^1(\tilde{L}_x^k(E))>0$, which in the $k=2$ case implies that $\Box(E)$ has positive Lebesgue measure.
\end{rem}

\subsection{Nonempty interior for lengths of non-degenerate $k$ chains}
The method above doesn't recover the full version of Theorem \ref{thm: interval} as it generally doesn't imply a strong enough nonempty interior result. It is indeed true that if the $k$ chain set has nonempty interior, then the method above implies that the same holds for its lengths set, and similarly for the box set. Unfortunately, the currently best known nonempty interior result for chains \cite{BIT} only guarantees a dimensional threshold $\frac{d+1}{2}$, much weaker than the one in Theorem \ref{thm: interval}.  

However, by combining the projection method with the sum set approach introduced in Section \ref{sumsetsection}, we can directly get nonempty interior result for $L^k(E)$, when $k\geq 2$. Note that the case $k=2$ was already taken care of by the sum set method alone in Remark \ref{lengthstwochains}.

Observe that for $k\geq 3$ and $x_1\in E$ fixed
\begin{equation}
    \begin{split}
        \{|x_0-x_1|&+|x_1-x_2|+\dots+|x_{k-1}-x_k|\colon x_0, x_2,x_3,\dots, x_k\in E\}\\
        =&\{|x_0-x_1|\colon x_0\in E\}\\
        &+\{|x_1-x_2|+|x_2-x_3|+\dots +|x_{k-1}-x_k|:\,x_2,x_3,\dots,x_k\in E\}
    \end{split}
\end{equation}
which contains the sum set 
$\Delta_{x_1}(E)+L^{k-1}_{x_1}(E)$. 

By the result in \cite{OT2020}, and the previous projection argument, we know there exists $x_1\in E$, such that, both $\Delta_{x_1}(E)$ and $L_{x_1}^{k-1}(E)$ have positive Lebesgue measure. Therefore, by Steinhaus theorem, $\Delta_{x_1}(E)+L^{k-1}_{x_1}(E)$ contains an interval. That implies that the (potentially degenerate) pinned length set 
$$\{|x_0-x_1|+|x_1-x_2|+\dots+|x_{k-1}-x_k|\colon x_0, x_2,x_3,\dots, x_k\in E\}$$
contains an interval. By slightly modifying this argument, we claim one can get nonempty interior for the non-degenerate pinned length set  
$$L^{k}_{x_1,1}(E)=\{|x_0-x_1|+|x_1-x_2|+\dots+|x_{k-1}-x_k|\colon x_0, x_2,x_3,\dots, x_k \text{ distinct points in } E\backslash\{x_1\}\}$$
where the sub-index $1$ indicates the pin is at the second point of the $k$-chain.

Applying Lemma \ref{lem: subsets} twice we can find $E_1,E_2,E_3$ pairwise separated subsets of $E$ with $\mu(E_i)>0$, where $\mu$ is a Frostman measure in $E$. Let $\mu_i$ a Frostman measure in $E_i$ with $\mu_i(B(x,r))\leq C_sr^s$ for $s>s_0$.
From the proof in \cite{OT2020}, one can check that 
$$\mu_3(\{x_1\in E_3\colon \mathcal{L}^1(\Delta_{x_1}(E_1))>0 \})>1/2$$
and
$$\mu_3(\{x_1\in E_3\colon \mathcal{L}^{k-1}(C_{x_1}^{k-1}(E_2)))>0 \})>1/2.$$
So one can find $x_1\in E_3$ such that both $\mathcal{L}^1(\Delta_{x_1}(E_1))>0$ and $\mathcal{L}^{k-1}(C_{x_1}^{k-1}(E_2))>0$. For such $x_3\in E_3$, we then have 
\begin{equation}
    \begin{split}
        \Delta_{x_1}(E_1)+L^{k-1}_{x_1}(E_2)\subset L^{k}_{x_1,1}(E)
    \end{split}
\end{equation}
where $\Delta_{x_1}(E_1)$ and $L^{k-1}_{x_1}(E_2)$ both have positive Lebesgue measure. Therefore, $L^{k}_{x_1,1}(E)$ and consequently $L^{k}(E)$ contain an interval.

We also observe that similar argument allows one to show the following. For any $k\geq 3$, $0<l<k$, and a point $x_l$ fixed, define
\begin{equation*}
    \begin{split}
       L^{k}_{x_l,l}(E):=&\{|x_0-x_1|+|x_1-x_2|+\dots +|x_{k-1}-x_k|\colon\\
       &x_0,x_1,\dots ,x_{l-1},x_{l+1},\dots x_k\text{ distinct points in } E\backslash \{x_l\}\}.
    \end{split}
\end{equation*}
Then there exists $x_{l}\in E$ such that $L_{x_l,l}^k(E)$ contains an interval. The idea is similar as above: in this situation, apart from non-degeneracy issues, one has $L^{k}_{x_l,l}(E)=L^{l}_{x_l}(E)+L_{x_l}^{k-l}(E)$.

\subsection{Perimeters of triangles}
We shall now illustrate how to apply this projection method in a situation that does not lend itself to the analysis employed in this paper up to this point. 

Let $T_2(E)$ denote the set of congruence classes of triangles with endpoints in $E \subset {\mathbb R}^d$, $d \ge 2$. It was shown in \cite{GILP2015} that if the Hausdorff dimension of $E \subset {\mathbb R}^2$ is greater than $\frac{8}{5}$, then the three-dimensional Lebesgue measure of $T_2(E)$ is positive. It was previously shown by Erdo\u{g}an and the second listed author (unpublished) that the critical threshold for this program is at least $\frac{3}{2}$. Applying the same method as above, we consider a map 
$$ (t_1,t_2,t_3) \mapsto (t_1+t_2+t_3, t_2, t_3)$$ and note that the Jacobian is bounded from above and below. We conclude that the set $P_2(E)$ of perimeters of triangles determined by triples of points in $E$ has a positive Lebesgue measure if the Hausdorff dimension of $E$ is greater than $\frac{8}{5}$. Corresponding results can be established in higher dimensions by the same method.

In the sequel, we are going to take on a systematic investigation of this method and its consequences. For example, there is a non-trivial result in \cite{GILP2015} about the $(k+1)$-dimensional Lebesgue measure of $T_k(E)$, the set of $k$-dimensional simplices determined by $E \subset {\mathbb R}^d$. If $d \ge 3$ and $k=3$, for example, it is interesting to ask if the set of areas of pyramids formed by the $4$-point configurations defining $T_3(E)$ has a positive Lebesgue measure. This will be one of many questions we will be investigating. 

\vskip.25in 

\section{$\Box(E)$ in vector spaces over finite fields}\label{sec: finite fields}

\vskip.125in 

The purpose of this section is to prove some analogs of the results of this paper in vector spaces over finite fields. Let ${\mathbb{F}}_q$ denote the finite field with $q$ elements, and let ${\mathbb{F}}_q^d$ denote the $d$-dimensional vector space over this field. Given $E \subset {\mathbb{F}}_q^d$, define 
$$ \Delta(E)=\{||x-y||: x,y \in E\}\subseteq \mathbb{F}_q,$$ where 
$$ ||x||=x_1^2+\dots+x_d^2.$$

The second listed author and Rudnev \cite{IR2007} proved that if $|E| \ge 3q^{\frac{d+1}{2}}$, then 
$$ \Delta(E)={\mathbb{F}}_q,$$ and Hart, the second listed author, Koh and Rudnev \cite{HIKR} proved that in odd dimensions this result cannot be improved in the sense that if $|E|=o(q^{\frac{d+1}{2}})$, then $|\Delta(E)|=o(q)$. The pinned version follows from \cite{BCCHIP2016}. A short argument is given below. 

In two dimensions, the exponent $\frac{3}{2}$ was improved to $\frac{4}{3}$ by Bennett, Hart, the second listed author, Pakianathan and Rudnev \cite{BHIPR}. The pinned version was proved in \cite{HLR16}. If $q$ is further assumed to be prime, the exponent was improved to $\frac{5}{4}$ by Murphy, Petridis, Pham, Rudnev and Stevens \cite{MPPRS}. 

In higher even dimensions, we do not know if the exponent $\frac{d+1}{2}$ can be improved. 

\vskip.125in 

Given $E \subset {\mathbb{F}}_q^d$, define 
$$ \Box(E)=\{||x-y||+||x-z||: x,y,z \in E,\, y\neq z\}.$$

\vskip.125in 

We have the following analog of our results in Euclidean space. 

\begin{thm} \label{ffanalog} Let $E$ be a subset of ${\mathbb{F}}_q^d$, $d \ge 1$, and let $\Box(E)$ be defined as above. Then \begin{itemize} 

\item i) If $d=1$ and $|E|\geq cq^{\frac{3}{4}}$, with $c$ sufficiently large, then $\Box(E)={\mathbb{F}}_q$, 

\vskip.125in 

\item ii) if $d=2$ and $|E| \ge cq^{\frac{4}{3}}$ with $c$ sufficiently large, then $\Box(E)={\mathbb{F}}_q$. 

\vskip.125in 

\item iii) if $d \ge 3$ and $|E| \ge cq^{\frac{d+1}{2}}$, then $\Box(E)={\mathbb{F}}_q$. 

\end{itemize} 

\end{thm}

\vskip.125in 
 
\begin{rem} Item iii) is where we have the most substantial difference between the finite field and Euclidean settings. In the Euclidean setting, there is a variety of improvements over the $\frac{d+1}{2}$ Falconer exponent, which we use as input in our results above. In the finite field case, the exponent $\frac{d+1}{2}$ is the best available in dimensions three and higher. Moreover, in odd dimensions, this exponent, and the corresponding estimates, are known to be sharp. \end{rem}  

To prove ii) and iii), observe that 
$$ \{||x-y||+||x-z||: y,z \in E\}=\Delta_x(E)+\Delta_x(E).$$ 

Just as in the Euclidean case, we can make sure that $y$ and $z$ above are never equal by pigeon-holing to disjoint subsets $E_1, E_2$ of $E$ such that $|E_i| \ge \frac{1}{3}|E|$. In the finite field case, this procedure is quite straightforward since we can simply divide $E$ into two sets of essentially the same size in an arbitrary manner. 

In the case of ii), the conclusion follows instantly from the pinned result in (\cite{HLR16} mentioned above. We need to work a tiny bit harder to prove iii). 

In the case of iii), the pinned version for $\Delta_x(E)$ is deduced from the known results for $\Delta(E)$ as follows. Define 
$$ \nu_x(t)=\sum_{||x-y||=t} E(y)$$ where $E(y)$ is the indicator function of the set $E$, and observe that 

$$ \sum_t \nu_x(t)=|E|.$$

Squaring this expression and applying Cauchy-Schwarz, we see that 

$$ {|E|}^2 \leq |\Delta_x(E)| \cdot \sum_t \nu_x^2(t). $$

Summing both sides with respect to $x \in E$, we obtain 
\begin{equation*}
    \begin{split}
        {|E|}^3 \leq& \sum_{x \in E} |\Delta_x(E)| \cdot \sum_t \nu_x^2(t)\\
        \leq& \max_{x \in E} |\Delta_x(E)| \sum_{x \in E} \sum_t \nu_x^2(t)
    \end{split}
\end{equation*}

\begin{equation} \label{hinge}=\max_{x \in E} |\Delta_x(E)| \cdot \sum_{t}\sum_{x,y,z} S_t(x-y)S_t(x-z) E(x)E(y)E(z),\end{equation}

where $S_t$ is the indicator function of 
$$ \{x \in {\mathbb{F}}_q^d: ||x||=t\}.$$

It was proved in \cite{BCCHIP2016} that 
$$ \sum_{x,y,z} S_t(x-y)S_t(x-z) E(x)E(y)E(z)={|E|}^3q^{-2}+D(E),$$ where 
$$ |D(E)| \leq \frac{4}{\log(2)} q^{\frac{d+1}{2}} \cdot \frac{{|E|}^2}{q^2}. $$

It follows that 

$$ \max_{x \in E} |\Delta_x(E)| \sum_{x \in E} \sum_t \nu_x^2(t) \leq  \max_{x \in E} |\Delta_x(E)| ({|E|}^3q^{-1}+q|D(E)|),$$ and the conclusion follows using the estimate on $D(E)$ above. 

\vskip.125in 

It remains to prove part i). For the sake of simplicity, we shall write down the proof in the case when $q=p^n$, with $p \equiv 3 \mod 4$. If $p \equiv 1 \mod 4$, the proof goes through by following the argument in \cite{BHIPR}. 

As before, given a $E\subset \mathbb{F}_q$, let $E_1, E_2$ be two disjoint subsets of $E$ with $|E_i|\geq \frac{1}{3}|E|$. Consider
$$ A=E_1 \times E_2, \ \text{and} \ B=\{(x,x): x \in E\}.$$

Define 
$$ \nu(t)=\sum_{x,y\in \mathbb{F}_q^2} A(x)B(y)S_t(x-y),$$ where $A,B$ are the indicator function of $A,B$, respectively, and $S_t$ is the indicator function of the circle 
$$ \left\{x \in {\mathbb{F}}_q^2: ||x|| \equiv x_1^2+x_2^2=t \right\}.$$ 

We have 
\begin{equation*}
    \begin{split}
        \sum_{t\in \mathbb{F}_q} \nu^2(t)=&\sum_{||x-y||=||x'-y'||} A(x)A(x')B(y)B(y')\\
        =&|A||B|+\sum_{\theta \in O_2({\mathbb{F}}_q)} \sum_{x'-y'=\theta(x-y); x \not=y; x' \not=y'} A(x)A(x')B(y)B(y')\\
        \leq& |A||B|+\sum_{\theta \in O_2({\mathbb{F}}_q); z \in \mathbb{F}^2_q} \lambda^A_{\theta}(z) \lambda^B_{\theta}(z),
    \end{split}
\end{equation*}

where $$ \lambda^A_{\theta}(z)=|\{(u,v) \in A \times A: u-\theta v=z \}|, $$ and similarly for $\lambda^B_{\theta}(z)$. 

\vskip.125in 

Observe that 
$$ \sum_{z} f(z) \lambda^A_{\theta}(z)=\sum_{u,v} f(u-\theta v)A(u)A(v).$$ 

The second formulation, while equivalent to the first, allows us to compute the Fourier transform of $\lambda^A_{\theta}$ and $\lambda^B_{\theta}$ in a simple way. 

\vskip.125in 

We have so far shown that 
$$ \sum_{t\in \mathbb{F}_q} \nu^2(t) \leq |A||B|+\sum_{\theta \in O_2(\mathbb{F}_q); z \in \mathbb {F}_q^2} \lambda^A_{\theta}(z)\lambda^B_{\theta}(z).$$ 

We have 
\begin{equation*}
    \begin{split}
        \sum_{\theta \in O_2(\mathbb {F}_q); z \in \mathbb {F}_q^2} \lambda^A_{\theta}(z) \lambda^B_{\theta}(z)
=&q^2 \sum_{\theta \in O_2(\mathbb {F}_q)} \sum_{m\in \mathbb{F}_q^2} \widehat{\lambda^A_{\theta}}(m) \overline{\widehat{\lambda^B_{\theta}}(m)}\\
=&q^6 \sum_{\theta \in O_2(\mathbb {F}_q)} \sum_{m\in \mathbb{F}_q^2} \widehat{A}(m){ \widehat{A}(\theta m)}  \overline{\widehat{B}(m)} \overline{\widehat{B}(\theta m)}\\
\leq&q^6 \cdot q^{-8} {|A|}^2 {|B|}^2 \cdot |O_2(\mathbb {F}_q)|\\
+&q^6 \sum_{\theta \in O_2(\mathbb {F}_q)} \sum_{m \not=(0,0)} |\widehat{A}(m) { \widehat{A}(\theta m)}  \overline{\widehat{B}(m)\widehat{B}(\theta m)}|\\
 \leq& q^6 \cdot q^{-8} {|A|}^2 {|B|}^2 \cdot |O_2(\mathbb {F}_q)|+q^6\sum_m {|\widehat{A}(m)|}^2 \sum_{\theta}{|\widehat{B}(\theta m)|}^2.
    \end{split}
\end{equation*}

The key is to estimate (when $m \not=(0,0)$)
$$ \sum_{\theta}  {|\widehat{B}(\theta m)|}^2.$$ 

This is equivalent to estimating 
$$ \sum_{||m||=t} {|\widehat{B}(m)|}^2 \ \text{with} \ t \not=0.$$ 

We have 
\begin{equation*}
    \begin{split}
        \sum_{||m||=t} {|\widehat{B}(m)|}^2=&q^{-4} \sum_{x,y} \sum_{m} \chi((x-y) \cdot m) S_t(m) B(x)B(y) \\
        =&q^{-2} \sum_{x,y} \widehat{S}_t(x-y)B(x)B(y)\\ 
=&q^{-2} \cdot q^{-2} \cdot |S_t| \cdot |B|+q^{-2} \sum_{x \not=y} \widehat{S}_t(x-y)B(x)B(y)\\
\leq& q^{-3} |B|+q^{-2} \sum_{x \not=y} 2 \cdot q^{-\frac{3}{2}} \cdot B(x)B(y)\\
\leq& q^{-3}|B|+2 \cdot q^{-2} \cdot q^{-\frac{3}{2}} \cdot {|B|}^2,
    \end{split}
\end{equation*}
where we have used the fact (Weil-Salie bound; see e.g. Lemma 2.2 in \cite{IR2007}) that 
$$ |\widehat{S}_t(m)| \leq 2q^{-\frac{3}{2}} \ \text{if} \ m \not=(0,0).$$

Plugging the estimate above back in, we see that 
$$ \sum_{t\in \mathbb{F}_q} \nu^2(t) \leq {|A|}^2 {|B|}^2 q^{-1}+2q^{\frac{1}{2}} |A| \cdot {|B|}^2+\text{much smaller terms}.$$ 

By Cauchy-Schwarz,
$$ {|A|}^2 {|B|}^2={\left( \sum_{t\in \mathbb{F}_q} \nu(t) \right)}^2 \leq |\Delta(A,B)| \cdot 2 \cdot {|A|}^2 \cdot {|B|}^2 \cdot q^{-1}$$ if
$$ |A|>2q^{\frac{3}{2}}.$$

It follows that 
$$ |\Box(E)|\geq |\Delta(A,B)| \ge \frac{q}{2} \quad \text{if} \ |A|> \sqrt{2} \cdot q^{\frac{3}{2}}.$$
Since $|A|=|E_1||E_2|\geq \frac{1}{9}|E|^{2}$, we have  $|A|>2q^{3/2}$ as long as $|E|>\sqrt{18}q^{3/4}$,
and the proof is complete.

\vskip.25in

\bibliographystyle{alpha}
\bibliography{sources}

\end{document}